\documentclass{amsart}
\usepackage{amsmath, calc}

\newtheorem{theorem}{Theorem}
\newtheorem*{lemma}{Lemma}
\newtheorem*{corollary}{Corollary}
\newtheorem*{euler}{Euler's Identity}

\theoremstyle{definition}
\newtheorem*{definition}{Definition}
\newtheorem*{example}{Example}

\theoremstyle{remark}
\newtheorem*{remark}{Remark}
\newtheorem*{convention}{Convention}

\begin{document}

\title{The remainders of the Euclidean algorithm with symmetric quotients}
\author{Barry Smith}
\email{barsmith@lvc.edu}
\address{Department of Mathematical Sciences \\ Lebanon Valley College\\
Annville, PA, USA}

\begin{abstract}
When the Euclidean algorithm produces a symmetric sequence of quotients, we give explicit formulas for the remainders that allow the analysis of two families of quadratic forms in the remainders.
\end{abstract}

\maketitle

 If $n$ is a prime number congruent to $1$ modulo $4$ and $a$ is a square root of $-1$ modulo $n$, then the simple continued fraction expansion of $\tfrac{n}{a}$ is symmetric.  In the mid-nineteenth century, Hermite, Serret, and Smith \cite{cH1848, jS1848, hS1855} used the convergents of this continued fraction to produce the unique integral representation $n = x^2 + y^2$.  More than a century later, Brillhart \cite{jB1972} noted that $x$ and $y$ also appear as the first two remainders smaller than $\sqrt{n}$ when the Euclidean algorithm is performed with $n$ and $a$, thereby providing a simple and efficient algorithm for computing the representation $n = x^2 + y^2$.  The algorithm also works when $n$ is composite \cite{HMW1990}, producing from the distinct pairs $\pm a$ of square roots of $-1$ modulo $n$ the distinct \emph{primitive} representations of $n$ as a sum of two squares.

The author and two students recently found a strikingly similar algorithm for computing multiplicative inverses in modular arithmetic \cite{DLS2014}.  If $n$ and $a$ are relatively prime positive integers with $n > 1$,  perform the Euclidean algorithm with $n^2$ and $an+1$.  The first remainder less than $n$ is a multiplicative inverse for $a$ modulo $n$.  Further, the sequence of quotients of the Euclidean algorithm with $n^2$ and $an+1$ is nearly symmetric, spoiled only by the middle two quotients differing by $2$. 

This algorithm was verified by finding explicit formulas for the remainders of the Euclidean algorithm with $n^2$ and $an+1$ in terms of the quotients and remainders of the Euclidean algorithm with $n$ and $a$. Theorem \ref{T:explicitremainders} of the present article gives analogous formulas for the remainders in Brillhart's algorithm.  In Theorem \ref{T:sumofsquaredremainders}, these formulas are used to give a family of quadratic forms in the remainders that all are multiples of $n$. Wagon \cite{sW1990} observed a special case: pairing off the remainders, the sum of the squares of each pair is a multiple of $n$. Theorem \ref{T:sumofsquaredremainders} reveals exactly which multiple.


\begin{definition}
Let $(u_1, \ldots, u_t)$ be a sequence of positive integers.  We define the symbol $\mathfrak{c}_{i,j}$ for $1 \leq i \leq j \leq t$ to be the tridiagonal determinant
\begin{equation*}
	\mathfrak{c}_{i,j} := 
	\begin{vmatrix} 
		u_i & 1 & 0 & \cdots & 0 & 0\\
		-1 & u_{i+1} & 1 & \cdots & 0 & 0\\
		0 & -1 & u_{i+2} & \cdots & 0 & 0\\
		\vdots & \vdots & \vdots & \ddots & \vdots & \vdots\\
		0 & 0 & 0 & \cdots & u_{j-1} & 1\\
		0 & 0 & 0 & \cdots & -1 & u_j
	\end{vmatrix}
\end{equation*}
We also define $\mathfrak{c}_{j+1,j} = 1$ for $0 \leq j \leq t$ and $\mathfrak{c}_{j+2,j} = 0$   for $-1 \leq j \leq t$.  The numbers $\mathfrak{c}_{i,j}$ are \emph{continuants} of the sequence $(u_1, \ldots, u_t)$.
\end{definition}

The most important property of continuants is Euler's identity (see \cite{gC1919} or \cite{GKP1989}):
\begin{euler}
If $(u_1, \ldots, u_t)$ is a sequence of positive integers, then for\\ $1 \leq i \leq  l \leq m+2$ and $m \leq s \leq t$,
\begin{align*}
	\mathfrak{c}_{i,s} \mathfrak{c}_{\, l,m} - \mathfrak{c}_{i,m} \mathfrak{c}_{\, l,s} = (-1)^{m-l+1} \mathfrak{c}_{i,l-2} \, \mathfrak{c}_{m+2,s} 
\end{align*}
\end{euler}

\noindent (Note: The requirements usually given are $1 \leq  i < l < m+2$ and $m<s \leq t$.  But with our convention that $\mathfrak{c}_{j+2,j}=0$, the cases where $i=l$, $l=m+2$, or $m=s$ are trivially true.)

Useful special cases include:
\begin{subequations}
\begin{align}
	\mathfrak{c}_{1,l-2} &= (-1)^{l+s} (\mathfrak{c}_{1,s} \mathfrak{c}_{l,s-1} - \mathfrak{c}_{1, s-1} \mathfrak{c}_{l, s})  &&\text{for $i=1$ and $m=s-1$} \label{E:zerothEulereq}\\
	\mathfrak{c}_{i,s} &= \mathfrak{c}_{i,m} \mathfrak{c}_{m+1,s} + \mathfrak{c}_{i,m-1} \mathfrak{c}_{m+2,s}  &&\text{for $l=m+1$}\label{E:firstEulereq}\\
	\mathfrak{c}_{i,s} &= u_i \mathfrak{c}_{i+1,s} + \mathfrak{c}_{i+2,s}  &&\text{setting $m=i$ in  \eqref{E:firstEulereq}}\label{E:secondEulereq}\\
	\mathfrak{c}_{i,s} &=  u_s \mathfrak{c}_{i,s-1} + \mathfrak{c}_{i,s-2}    &&\text{setting $m=s-1$ in  \eqref{E:firstEulereq}}\label{E:thirdEulereq}
\end{align}
\end{subequations}

We shall only have cause to consider the Euclidean algorithm when the sequence of quotients is symmetric.  In this case, if the initial pair of integers is $n$, $a$, we set $n=r_1$, $a=r_2$, and write the Euclidean algorithm as 
\begin{align}\label{E:EucAlg}
r_1 &= \hspace{0.2cm} q_1 r_2 \hspace{-3.5cm} &&+ r_3 \notag \\ \notag
r_2 &= \hspace{0.2cm} q_2r_3 \hspace{-3.5cm} &&+ r_4\\ \notag
&\mathrel{\makebox[\widthof{=}]{\vdots}} \hspace{-3.5cm} &&\mathrel{\makebox[\widthof{+}]{\vdots}} \\ \notag
r_{s} &= \hspace{0.1cm} q_{s}r_{s+1} \hspace{-3.5cm} &&+ r_{s+2}\\ \notag
r_{s+1} &= \hspace{0.1cm} q_s r_{s+2} \hspace{-3.5cm} &&+ r_{s+3}\\ \notag
&\mathrel{\makebox[\widthof{=}]{\vdots}} \hspace{-3.5cm} &&\mathrel{\makebox[\widthof{+}]{\vdots}} \\ \notag
r_{2s-1} &= \hspace{0.2cm} q_2 r_{2s} \hspace{-3.5cm} &&+ r_{2s+1}\\ \notag
r_{2s} &=  \hspace{0.1cm} q_1 r_{2s+1} \hspace{-3.5cm} &&+ r_{2s+2},
\end{align}
where $r_{2s+1} = 1$ and $r_{2s+2} = 0$.  

\begin{convention}
Because the remainders must decrease, the final quotient $q_1$ cannot be $1$.  The following convention will allow us to consider $q_1=1$ as well:  if the first quotient of the Euclidean algorithm with $n$ and $a$ is $1$, rewrite the final equation $r_{l} = q_{l} \cdot 1 + 0$ as the two equations
\begin{align*}
	r_l &= (q_l - 1) \cdot 1 + 1\\
	1 &= 1 \cdot 1 + 0
\end{align*}
This allows the possibility for the Euclidean algorithm to produce a symmetric sequence of quotients with $q_1 = 1$.  In the results to follow, any discussion of the Euclidean algorithm will allow this possibility.  All proofs will be valid in this more general setting because the knowledge that the remainders decrease is never used.
\end{convention}

\begin{lemma}\label{L:endofEA}
Let $n$ and $a$ be relatively prime positive integers and suppose the sequence of quotients of the Eulcidean algorithm with $n$ and $a$ ends $\left( \ldots q_s, q_{s-1} \ldots, q_2, q_1 \right)$.  If $\mathfrak{c}_{i,j}$ are continuants for the sequence $(q_1, \ldots, q_s)$, then the sequence of remainders ends $ \left( \ldots \mathfrak{c}_{1,s}, \mathfrak{c}_{1,s-1}, \ldots, \mathfrak{c}_{1,0}, \mathfrak{c}_{1,-1} \right)$.
\end{lemma}

\begin{proof}
This can be proved by induction, noting from Equation \eqref{E:thirdEulereq} that the remainders satisfy the same recursion as the continuants.
\end{proof}

The lemma shows that formulas for the remainders in the Euclidean algorithm are already given by the continuants of the full sequence of quotients.  The combination of this with the symmetry of the quotients formed the crux of Brillhart's proof \cite{jB1972}.  The following theorem gives more complicated expressions for the remainders involving continuants of only \emph{half} of the symmetric sequence of quotients.  The added complication is justified by Theorem 2 in which these expressions quickly reveal families collections of quadratic forms in the remainders that are all multiples of $n$.  The formulas are analogous to those in\cite{DLS2014} for the remainders  when performing the Euclidean algorithm with $n^2$ and $an+1$.  Perhaps both are special cases of some more general pattern. 

\begin{theorem}\label{T:explicitremainders}
Assume that the Euclidean algorithm with the relatively prime positive integers $n$ and $a$ produces a symmetric sequence of quotients $q_1, \ldots, q_s, q_s, \ldots, q_1$.
If $\mathfrak{c}_{i,j}$ are continuants of the sequence $\left( q_1, \ldots, q_s \right)$, then for $1 \leq i \leq s+1$,
\begin{align*}
	r_i &= \mathfrak{c}_{1,s} \mathfrak{c}_{i,s} + \mathfrak{c}_{1,s-1} \mathfrak{c}_{i,s-1} \quad \text{ and }\\
	r_{2s-i+3} = \mathfrak{c}_{1,i-2} &= (-1)^{i+s} \left( \mathfrak{c}_{1,s} \mathfrak{c}_{i,s-1} - \mathfrak{c}_{1,s-1} \mathfrak{c}_{i,s} \right)
\end{align*}
\end{theorem}

\begin{proof}
The lemma implies that for $1 \leq i \leq s+1$ the remainder $r_{2s-i+3}$ is $\mathfrak{c}_{1,i-2}$.  The second equality above then follows from Equation \eqref{E:zerothEulereq}.  The lemma also implies that 
\begin{equation*}
	r_{s+1} = \mathfrak{c}_{1,s} = \mathfrak{c}_{1,s} \mathfrak{c}_{s+1,s} + \mathfrak{c}_{1,s-1} \mathfrak{c}_{s+1,s-1}
\end{equation*}
The $s$th equation of the Euclidean algorithm with $n$ and $a$ is therefore
\begin{align*}
	r_{s} &= q_s \mathfrak{c}_{1,s} + \mathfrak{c}_{1,s-1}\\
	&= \mathfrak{c}_{1,s} \mathfrak{c}_{s,s} + \mathfrak{c}_{1,s-1} \mathfrak{c}_{s,s-1}
\end{align*}
Thus, the formula $r_i = \mathfrak{c}_{1,s} \mathfrak{c}_{i,s} + \mathfrak{c}_{1,s-1} \mathfrak{c}_{i,s-1}$ is valid for $i=s$ and $i=s+1$.  To prove the other cases of the formula, we use induction.  We must show
\begin{equation*}
	\mathfrak{c}_{1,s} \mathfrak{c}_{i,s} + \mathfrak{c}_{1,s-1} \mathfrak{c}_{i,s-1} = q_i \left( \mathfrak{c}_{1,s} \mathfrak{c}_{i+1,s} + \mathfrak{c}_{1,s-1} \mathfrak{c}_{i+1,s-1} \right) + \mathfrak{c}_{1,s} \mathfrak{c}_{i+2,s} + \mathfrak{c}_{1,s-1} \mathfrak{c}_{i+2,s-1}
\end{equation*}
for $1 \leq i \leq s-1$.  This follows from Equation \eqref{E:secondEulereq}.
\end{proof}

Theorem \ref{T:explicitremainders} yields in particular the formula 
\begin{equation*}
	n = r_1 = \mathfrak{c}_{1,s-1}^2 + \mathfrak{c}_{1,s}^2 = r_{s+1}^2 + r_{s+2}^2,
\end{equation*}
which is Brillhart's main result \cite{jB1972}.  It is known \cite{oP1954} that the Euclidean algorithm with relatively prime $n$ and $a$ will produce symmetric sequence of quotients if and only if $a^2 \equiv -1 \pmod{n}$ (applying the convention stated above the Lemma).  Thus, we also recover from Theorem \ref{T:explicitremainders} the Two Squares Theorem of Fermat: the prime numbers that are the sum of two integer squares are those that are congruent to 1 modulo 4.

Now suppose we are given a sequence of positive integers and wish to produce it as the sequence of quotients when the Euclidean algorithm is performed (adopting the convention stated above the Lemma).  This will occur for a unique input pair of relatively prime integers: the reduced numerator and denominator of the simple continued fraction with the given sequence as partial quotients.  We represent by $[q_1, \ldots, q_n]$ the simple continued fraction with sequence of partial quotients $q_1, \ldots, q_n$. We will also use the convention that 1 is the rational number corresponding to the simple continued fraction with empty list of partial quotients.


The next theorem generalizes the formula $n = r_{s+1}^2 + r_{s+2}^2$ to other quadratic forms in the remainders.







\begin{theorem}\label{T:sumofsquaredremainders}
Assume that the Euclidean algorithm with the relatively prime positive integers $n$ and $a$ produces a symmetric sequence of quotients $q_1, \ldots, q_s, q_s, \ldots, q_1$.  If $i$ and $j$ are integers with $0 \leq j \leq i \leq s$, then
\begin{align*}
	 r_{i-j+1} r_{i+j+1} + r_{2s+2-i-j} r_{2s+2-i+j} &=n \cdot \mathfrak{r}_{i,j}^+\\
	 r_{i-j+1} r_{i+j+2} - r_{2s+1-i-j} r_{2s+2-i+j} &= n \cdot \mathfrak{r}_{i,j}^-
\end{align*}
where $\mathfrak{r}_{i,j}^+$ and $\mathfrak{r}_{i,j}^-$ are the $(2j+1)$th and $(2j+2)$th remainders respectively  when the Euclidean algorithm is performed with the reduced numerator and denominator of the symmetric continued fraction $[q_{i-j+1},\ldots,q_{s}, q_s, \ldots, q_{i-j+1}]$.
 
%
\end{theorem}

\begin{remark}
Wagon \cite{sW1990} observed that for $0 \leq i \leq s$, the number $r_{i+1}^2 + r_{2s+2-i}^2$ is a multiple of $n$.  We recover this by letting $j=0$ in the first equation of the theorem.
\end{remark}

\begin{proof}

Let the numbers $\mathfrak{c}_{i,j}$ be the continuants of the sequence $(q_1, \ldots, q_s)$.  If $i + j \leq s$, then by Theorem \ref{T:explicitremainders} we have
\begin{align*}
	&\phantom{=} r_{i-j+1} r_{i+j+1} + r_{2s+2-i-j} r_{2s+2-i+j} \\
	&= \left( \mathfrak{c}_{1,s} \mathfrak{c}_{i-j+1,s} + \mathfrak{c}_{1,s-1} \mathfrak{c}_{i-j+1,s-1} \right) \cdot \left( \mathfrak{c}_{1,s} \mathfrak{c}_{i+j+1,s} + \mathfrak{c}_{1,s-1} \mathfrak{c}_{i+j+1,s-1} \right) \\
	&+ (-1)^{i+j+s+1} \left( \mathfrak{c}_{1,s} \mathfrak{c}_{i+j+1,s-1} - \mathfrak{c}_{1,s-1} \mathfrak{c}_{i+j+1,s} \right) \cdot (-1)^{i-j+s+1} \left( \mathfrak{c}_{1,s}  \mathfrak{c}_{i-j+1,s-1} - \mathfrak{c}_{1,s-1} \mathfrak{c}_{i-j+1,s} \right) \notag \\
	&= \left( \mathfrak{c}_{1,s-1}^2 + \mathfrak{c}_{1,s}^2 \right) \cdot \left( \mathfrak{c}_{i-j+1,s-1} \mathfrak{c}_{i+j+1,s-1} + \mathfrak{c}_{i-j+1,s} \mathfrak{c}_{i+j+1,s} \right)
\end{align*}
Theorem \ref{T:explicitremainders} shows that the first factor is $n$.  Letting $[q_{i-j+1}, \ldots, q_s, q_s, \ldots, q_{i-j+1}]$ be the symmetric sequence in the theorem, we find that the second factor is $\mathfrak{r}_{i,j}^+$

One may perform a similar computation when $i + j > s$. The second equation above may be verified by similar computations in the separate cases $i+j < s$ and $i+j \geq s$.
\end{proof}

\begin{example}
Setting $n=829$ and $a=246$, the Euclidean algorithm produces the symmetric sequence of quotients $3$, $2$, $1$, $2$, $2$, $1$, $2$, $3$ and sequence of remainders
\begin{equation*}
	829 \qquad 246 \qquad 91 \qquad 64 \qquad 27 \qquad10 \qquad 7 \qquad 3 \qquad 1 \qquad 0
\end{equation*}
The symmetric subsequence of quotients $2$, $1$, $2$, $2$, $1$, $2$ is produced by performing the Euclidean algorithm with $73$ and $27$, which yields the sequence of remainders
\begin{equation*}
	73 \qquad 27 \qquad 19 \qquad 8 \qquad 3 \qquad 2 \qquad 1 \qquad 0
\end{equation*}
The symmetric sequence of quotients $1$, $2$, $2$, $1$ is produced by performing the Euclidean algorithm with $10$ and $7$, which yields the sequence of remainders
\begin{equation*}
	10 \qquad 7 \qquad 3 \qquad 1 \qquad 1 \qquad 0
\end{equation*}
(using the convention stated above the Lemma.)  Finally, the symmetric sequence of quotients $2$, $2$ is produced by performing the Euclidean algorithm with $5$ and $2$, yielding the sequence of remainders
\begin{equation*}
	5 \qquad 2 \qquad 1 \qquad 0
\end{equation*}

The following factorizations are all consequences of Theorem \ref{T:sumofsquaredremainders}:
\begin{align*}
	246^2 + 1^2 &= 73 \cdot 829\\
	246\cdot91 - 3\cdot1 &= 27 \cdot 829\\
	246\cdot64 + 7\cdot1 &= 19 \cdot 829\\
	246\cdot27 - 10\cdot1 &= 8 \cdot 829\\
	246\cdot10 + 27\cdot1 &= 3 \cdot 829\\
	246\cdot7-64\cdot1 &= 2 \cdot 829\\
	246\cdot3 + 91\cdot1 &= 1 \cdot 829\\
	91^2 + 3^2  &= 10 \cdot 829\\
	91\cdot64-7\cdot3 &= 7 \cdot 829\\	
	91\cdot27+10\cdot3  &= 3\cdot829\\
	91\cdot10-27\cdot3 &= 1 \cdot 829\\
	91\cdot7 + 64\cdot3 &= 1 \cdot829\\ 	
	64^2 + 7^2 &= 5 \cdot 829\\
	64\cdot27 - 10\cdot7 &= 2 \cdot 829\\
	64\cdot10 + 27\cdot7 &= 1 \cdot 829\\
	27^2 + 10^2  &= 1 \cdot 829\\
\end{align*}
\end{example}

The classical observation that the Euclidean algorithm with relatively prime $n$ and $a$ will produce symmetric sequence of quotients only if $a^2+1$ is a multiple of $n$ is recovered by setting $i=1$ and $j=0$ in Theorem \ref{T:explicitremainders}, as is a little more:

\begin{corollary} 
Assume that the Euclidean algorithm with the relatively prime positive integers $n$ and $a$ produces a symmetric sequence of quotients $q_1, \ldots, q_s, q_s, \ldots, q_1$. Then the  reduced numerator of the continued fraction $[q_2, \ldots, q_s, q_s, \ldots, q_2]$ is the integer $\mathfrak{r}_{2,0}^+ = \tfrac{a^2+1}{n}$, and the denominator is the remainder when $a$ is divided by $\mathfrak{r}_{2,0}^+$.
\end{corollary}
\begin{proof}
Setting $i=2$ and $j=0$ in Theorem \ref{T:sumofsquaredremainders} gives
\begin{equation*}
	a^2 + 1 = r_2^2 + r_{2s+1}^2 = n \cdot \mathfrak{r}_{2,0}^+,
\end{equation*}
which proves everything except the description of the denominator.  But
\begin{align*}
	\left[ q_2, \ldots, q_s, q_s, \ldots, q_1 \right] &= \frac{a}{n - q_1 a}
\end{align*}
Since the numerator of a continued fraction is unchanged upon reversing its sequence of quotients and since the denominator of $[q_1, \ldots, q_s, q_s, \ldots, q_2]$ is the numerator of $[q_2, \ldots, q_s, q_s, \ldots, q_2]$,
\begin{equation*}
	[q_1, \ldots, q_s, q_s, \ldots, q_2] = \frac{a}{\mathfrak{r}_{2,0}^+}.
\end{equation*}
Thus, 
\begin{equation*}
	[q_2, \ldots, q_s, q_s, \ldots, q_2] = \frac{\mathfrak{r}_{2,0}^+}{a - q_1 \mathfrak{r}_{2,0}^+},
\end{equation*}
and $a- q_1 \mathfrak{r}_{2,0}^+$ is the remainder when $a$ is divided by $\mathfrak{r}_{2,0}^+$ since 
\begin{equation*}
	\left\lfloor \tfrac{a}{\mathfrak{r}_{2,0}^+} \right\rfloor = \left\lfloor \tfrac{na}{a^2+1} \right\rfloor = \left\lfloor \tfrac{n}{a} \right\rfloor. 
\end{equation*}




\end{proof}

Given the numerator and denominator of the continued fraction $[q_1, \ldots, q_s, q_s, \ldots, q_1]$, we may apply this corollary iteratively to quickly compute the reduced fractions $[q_i, \ldots, q_s, q_s, \ldots, q_i]$ for $i=2, \ldots, s$.  In the example above:
\begin{equation*}
	[3,2,1,2,2,1,2,3] = \tfrac{829}{246}
\end{equation*}
We compute $\tfrac{246^2+1}{829} = 73$, whence 
\begin{equation*}
	[2,1,2,2,1,2] = \tfrac{73}{246 - 3 \cdot 73} = \tfrac{73}{27}.
\end{equation*}
Then $\tfrac{27^2+1}{73} = 10$, whence 
\begin{equation*}
	[1,2,2,1] = \tfrac{10}{27-2\cdot10} = \tfrac{10}{7}.
\end{equation*}
Finally, $\tfrac{7^2+1}{10} = 5$, whence 
\begin{equation*}
	[2,2] = \tfrac{5}{7-1\cdot5} = \tfrac{5}{2}.
\end{equation*}  

\end{document}